\documentclass{amsart}
\usepackage{amsmath,amsthm, amscd}
\usepackage{amsfonts}
\usepackage{txfonts}
\usepackage{marvosym}
\usepackage{graphicx}
\usepackage{color}
\usepackage{setspace}
\theoremstyle{plain}
\newtheorem{thm}{Theorem}

\theoremstyle{definition}

\newtheorem*{conj}{Conjecture}
\newtheorem{rem}{Remark}

\newtheorem{lem}{Lemma}

\newtheorem*{cor}{Corollary}
\newcommand{\comment}[1]{}

\newcommand{\R}{\mathbb{R}}
\newcommand{\N}{\mathbb{N}}

\newcommand{\G}{\mathcal{G}}

\newcommand{\textd}{\textit}

 \usepackage{float} 
\begin{document}
\vspace*{-0.9cm}

\title{Toroidal embeddings of planar graphs are knotted or linked}

\author{Senja Barthel, Dorothy Buck}
\address{Department of Mathematics, Imperial College London, London, SW7 2AZ, United Kingdom}
\email{s.barthel11@imperial.ac.uk}

\begin{abstract}
We give explicit deformations of embeddings of planar graphs that lie on the standard  torus~$T^2 \subset \R^3$ and that contain neither a nontrivial knot nor a nonsplit link into the plane. It follows that ravels do not embed on the torus. Our results provide general insight into properties of molecules that are synthesized on a torus.
\end{abstract}
\maketitle
\section{Introduction}
\label{intro}
The interaction between mathematical topology, in particular topological graph theory, and the investigation of chemical structures is a rich area (\cite{SauvageAmabilino}-\cite{Flapan} and the references therein). Topological graph theory studies embeddings of graphs in $3$-space. In most cases, these spatial graphs can be thought of as knots or links with additional edges attached: A \textbf{spatial graph} is the image $\mathcal{G}$ of an embedding $f:G\rightarrow \R^3$, where $G$ is a graph. Two spatial graphs are considered different if it is not possible to transform one into the other without self-intersections during the transformation. (Note that a graph that contains a cycle has many different spatial graph realisations.) The allowed transformations are \textbf{ambient isotopies} such as bending, stretching and shrinking of edges as long as no edge is shrunk to a point.\\
Spatial graphs can model molecules. The typical example are embeddings of molecular graphs. These are spatial graphs whose vertices are placed at the positions of atoms and edges are added between two vertices if a bond is formed  between their corresponding atoms. A different example of spatial graphs describing molecules is the representation of a coordination polymer where its ligands correspond to the edges of the spatial graph and the coordination entities correspond to its vertices. Results about spatial graphs directly translate to information about the configuration of molecules. In particular, if entangled chemical structures like knots, links, braids, and ravels are present, topological graph theory can be an appropriate framework \cite{SauvageAmabilino}-\cite{Flapan}. Knot theoretical methods can predict, or give constraints on, the possible entanglements and related properties like chirality of chemical structures.\\
As molecules with non-standard topological structure often have unusal chemical properties, synthetic organic chemists have designed new structures that include entanglements (e.g.~\cite{Ayme}-\cite{Simmons}). Furthermore, crystal engineers have produced coordination networks that contain knots and links~\cite{CarlucciProserpioReview}. Many $3$-dimensional and several $2$-dimensional entangled structures have been reported by experimentalists~\cite{3Dclassification},~\cite{3Dclassification2}. A concept of topological entanglements called ravels that are not caused by knots or links was introduced to chemistry by Castle~et~al.~\cite{Ravels}. Following, a mathematical description in terms of spatial graphs of one family of ravels was given by Farkas~et.~al.~\cite{Farkas}. A \textbf{ravel} is a nontrivial embedded $\theta_n$-graph that does not contain a nontrivial knot. The $\theta_{n}$-graph consists of two vertices that are joined by $n$ edges; it cannot contain links. A molecular ravel was synthesised by Clegg~et~al.~\cite{Lindoy}. \\
Many molecules have a corresponding \textbf{planar} graph, i.e., are described by a graph for which an embedding on the sphere~$S^2$ (equivalently on the plane~$\R^2$) exists. Such an embedding is a \textbf{trivial embedding} and its image is a \textbf{trivial spatial graph}. Embeddings of molecules on a sphere (respectively trivial spatial graphs) contain no nontrivial knots or nonsplit links. \\
We are interested in molecules that are described by a planar graph and that are flexible enough to be realised in \textit{topologically different forms} in $3$-space, in particular as knotted molecules. Examples of such flexible molecules are carcerands or molecules that are built with DNA strands. The next more topologically complex surface in $\R^3$ after the sphere is the torus, and embeddings on the standard torus can be nontrivially knotted and linked. It is therefore reasonable to investigate how molecules with planar underlying graphs can embed on the torus to analyse the next level of complexity of their realisation. We call spatial graphs (and molecules that are described by them) on the torus that cannot be transformed to lie on the sphere \textbf{toroidal}. Whenever we use the word torus in this paper, we refer to the standard torus.\\
To state the following conjecture that was given by Castle, Evans and Hyde~\cite{Hyde} we make the following definitions:
A \textbf{polyhedral} molecule has an underlying graph that is planar 3-connected and simple. A graph is \textbf{\textit{n}-connected} if at least $n$ vertices and their incident edges have to be removed to disconnect the graph or to reduce it to a single vertex. A graph is \textbf{simple} if it has neither multiple edges between a given pair of vertices nor loops from a vertex to itself.
\begin{conj}[Castle, Evans and Hyde~\cite{Hyde}]
All polyhedral toroidal molecules contain a nontrivial knot or a nonsplit link.
\end{conj}
The main result (Theorem~\ref{goal}) proves this conjecture without assuming 3-connectivity or simpleness using topological graph theory. The argument gives an explicit deformation from embeddings of planar graphs on the torus that contain neither a nontrivial knot nor a nonsplit link into the plane. A much shorter but less intuitive proof that relies on deep theorems of topological graph theory is given in~\cite{short}. The argument on hand not only presents a self-sufficient argument but will hopefully also give the reader a better feeling for the nature of graphs that are embedded on the torus.

\begin{thm}[Existence of knots and links]\label{goal}
Let $G$ be a planar graph and $f:G\rightarrow \R^3$ be an embedding of $G$ with image~$\mathcal{G}$. If $\mathcal{G}$ lies in the torus~$T^2$ and contains no subgraph that is a nontrivial knot or a nonsplit link, then the embedding $f$ is trivial.
\end{thm}
\begin{rem}\label{AusGoal}
The above theorem can be reformulated as
\begin{itemize}
\item Let $\G$ be a nontrivial embedding of a planar graph that contains neither a nontrivial knot nor a nonsplit link. Then the minimal genus of $\G$ is strictly greater than one.\\
\item On the torus exist no nontrivial embeddings of planar graphs that neither contain a nontrivial knot nor a nonsplit link.
\end{itemize}
\end{rem}

\begin{cor}[Ravels do not embed on the torus]
Every nontrivial embedding of a $\theta_{n}$-graph on the torus contains a nontrivial knot.
\end{cor}

The topological structure of the surface on which the spatial graph is embedded is crucial for the theorem. For all closed orientable surfaces of genus $g > 1$, there exist examples of planar spatial graphs that are neither knotted or linked nor embeddable on a closed orientable surface with genus less than $g$. A famous example for the closed oriented genus two surface is Kinoshita's $\theta$-curve~\cite{Kinoshita} (Figure~\ref{KinoshitaThetaSurface}). Kinoshita's $\theta$-curve is a ravel. Ravels give more examples, since a ravel is by definition described by the embedding of a planar graph that is nontrivial although it contains no nontrivial knot or nonsplit links, but that does not embed on the torus by Theorem~\ref{goal}. As every spatial graph embeds on a compact closed surface of some genus, it follows that ravels correspond to spatial graphs that are neither trivial embedded nor knotted or linked but which are realisations of planar graphs on higher genus surfaces. One needs arbitrarily high genera to accommodate all ravels which can be shown using a Borromean construction as given by Suzuki~\cite{Suzuki}.
\begin{figure}[H]
	\centering
	\def\svgwidth{300pt}
	 \input{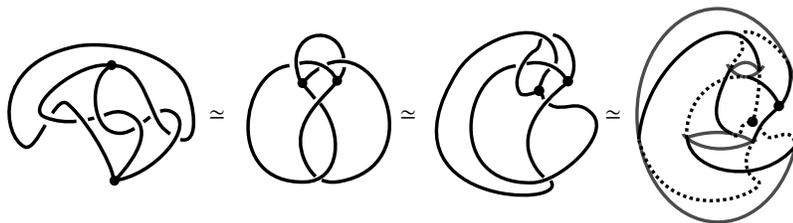}
\caption{Kinoshita's $\theta$-curve. Kinoshita's $\theta$-curve can be embedded on the closed connected surface of genus two. This is its minimal genus since it does not embed on the torus by Remark~\ref{AusGoal}.}
\label{KinoshitaThetaSurface}
\end{figure}

Castle, Evans and Hyde~\cite{Hyde} proved that polyhedral toroidal molecules which contain a nontrivial knot are chiral. The chirality of polyhedral toroidal molecules which contain a nonsplit link is shown in~\cite{chirality}. Note that topological chirality implies chemical chirality. I.e., if the spatial graph describing a molecule is chiral as topological object, it follows that there is also no chemical transformation that deforms the molecule to its mirror image. This is due to chemically realisable transformations being more restrictive than ambient isotopies.

\begin{thm}[Chirality \cite{Hyde}, \cite{chirality}] \label{chiral} 
Let $G$ be a simple $3$-connected planar graph and $f:G\rightarrow T^{2} \subset \R^3$ be an embedding of $G$ with image~$\mathcal{G}$ on the torus $T^{2}$. If $\mathcal{G} \subset T^2$ is nontrivial embedded, then $\mathcal{G}$ is topologically chiral in~$R^3$.
\end{thm}

\subsection*{Acknowledgements}
The first author thanks Tom Coates, Erica Flapan, Youngsik Huh, Stephen Hyde, Danielle O'Donnol, Makoto Ozawa, and Kouki Taniyama for helpful comments and discussions and is very grateful for the extensive comments of the first referee that improved the quality of the paper a lot. Matt Rathbun is particularly thanked for long discussions and the outline of the second step in the proof. The research was financially supported by the Roth studentship of Imperial College London mathematics department, the DAAD, the Evangelisches Studienwerk, the Doris Chen award, and by a JSPS grant awarded to Kouki Taniyama. The second author was partially supported by EPSRC Grants EP/G039585/1 and EP/H031367/1, and the Leverhulme Trust Grant RP2013-K-017.

\section{Proof of the theorem by giving an explicit isotopy}
\subsection{Outline of the proof of Theorem~\ref{goal} and preparations.}

The idea of the proof of Theorem~\ref{goal} is the following: Let~$G$ be a planar graph and $\mathcal{G}$ be the spatial graph that is the image of the embedding $f:G \rightarrow T^2$. Assume that $\mathcal{G}$ contains neither a nontrivial knot nor a nonsplit link. We give a general construction for an explicit ambient isotopy in $\R^3$ from the spatial graph~$\mathcal{G}$ to a trivial spatial graph~$\mathcal{G}'$ (which is embedded in the plane~$\R^2 \subset \R^3$). This demonstrates that any embedding of a planar graph that is embedded on the torus and contains neither a nontrivial knot nor a nonsplit link is trivial. The ambient isotopy is illustrated in Figure~\ref{Fig20}.\\
To construct the ambient isotopy in $\R^3$ that transforms the embedding~$\mathcal{G} \subset T^2$ of a planar graph~$G$ that contains neither a nontrivial knot nor a nonsplit link to a trivial spatial graph~$\mathcal{G}'$, we first note using Lemma~\ref{conn} that it is sufficient to consider connected graphs. Step~(1) of the proof shows that it is furthermore sufficient to restrict to spatial graphs that contain a subgraph of the form~$T(1,n), n>0$ (defined below) since otherwise, the graphs would be nonplanar or would already be trivially embedded. For these graphs, the desired ambient isotopy consists of three deformations. The first one is a twist around the core of the torus which transforms~$T(1,n)$ into the longitude~$l$ ($T(1,n)$ and $l$ are marked red in Figure~\ref{Fig20}). This twist is described in step~(3). Step~(2) of the proof is a technicality that ensures the existence of the twist given in step~(3) by arguing that there exists a meridian of $T^2$ that intersects $\mathcal{G}$ in one point only. As~$G$ is planar by assumption, it follows from Theorem~\ref{Tutte} that the conflict graph of~$G$ with respect to~$l$ is bipartite. The bipartiteness of the conflict graph is used together with the property of the spatial graph being embedded on the torus to show in step~(4) and step~(5) of the proof that a second ambient isotopy can be performed. This ambient isotopy is given in step~(6). It rotates the spatial fragments in space around the longitude~$l$, so that $\mathcal{G}$ is ambient isotoped to a spatial graph that is embedded on the surface $8 \times S^1$, such that conflicting spatial fragments lie in different components of $(8 \times S^1) \setminus (P \times S^1)$. We use $8$ to denote a loop with one point~$P$ of self-intersection.
Step~(7) shows that there is an individual ambient isotopy for each spatial fragment that transforms it in space to a trivial spatial fragment, independently from all other spatial fragments.
The combination of the ambient isotopies given in step~(3), step~(6) and step~(7) gives the desired deformation of $\mathcal{G}$ into the plane which proves Theorem~\ref{goal}. 
\begin{figure}[h]
	\centering
	\def\svgwidth{420pt}
	 \input{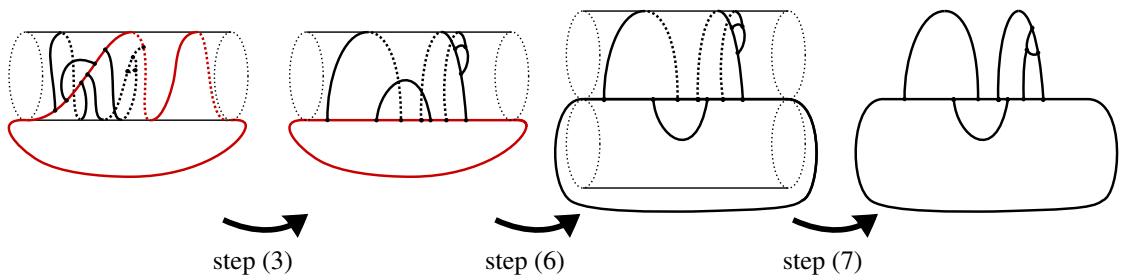}
\caption{The ambient isotopy of the proof in Section~\ref{proof}. }
\label{Fig20}
\end{figure}

To work on the torus, we define the following: A \textbf{meridian} of a solid torus~$T$ is an essential simple closed curve in $\partial T$ that bounds a disc in $T$. (An essential simple closed curve in $\partial T$ does not bound a disc in $\partial T$.) The \textbf{preferred longitude} is the simple closed curve in $\partial T$ that intersects the meridian once and has linking number zero with the core of the torus~$T$. For the standard torus~$T^2$ in $\R^3$ define the meridian and preferred longitude analogously by taking the interior of the solid torus to be the bounded component of the complement of $T^2$ in $\R^3$. Whenever we use the term \textbf{longitude} in this paper, we will refer to the preferred longitude.\\
To describe knots and links on the torus, we define a \textbf{torus knot} or \textbf{torus link} to be a knot or link that is embedded on the standard torus $T^2$ following the longitude of the standard torus $p$ times and the meridian $q$ times. Those knots or links are denoted by \textit{\textbf{T(p,q)}}  with $p,q$ integers. Therefore, a meridian is denoted by $m=T(0,1)$ and a longitude by $l=T(1,0)$. An unknot that bounds a disc in $T^2$ is denoted by $T(0,0)$. \\
The concept of the `conflict graph' defined below will be needed (compare Figure~\ref{conflictgraph}), to ensure planarity of the considered graphs during the proof:\\
A \textbf{cycle} is a simple closed path in a graph or in a spatial graph. Let $C$ be a cycle in a graph $G$. The connected components $f_i$ of the graph $G\setminus C$ are called  \textbf{fragments} of $G$ with respect to $C$ and two fragments $f_i$ and $f_j$ \textbf{conflict} if at least one of the following conditions is satisfied:  

\begin{itemize}
\item
There exist three points on $C$ to which both components $f_i$ and $f_j$ are attached to. 
\item
There exist four \textbf{interlaced} points $v_1, v_2, v_3, v_4$ on $C$ in cyclic order so that $f_i$ is attached to $C$ at $v_1$ and $v_3$ and $f_j$ is attached to $C$ at $v_2$ and $v_4$.
\end{itemize}

Let $f_i$ be a fragment of $G$ with respect to $C$. Call the set of points in which $f_i$ is attached to $C$ the \textd{endpointset} $v(f_i)$ of the fragment~$f_i$.  Note that vertices of $C$ can be elements of different endpointsets $v(f_i)$ and $v(f_j)$ where $f_i$ and $f_j$ might or might not conflict. The restriction of a spatial graph to a fragment is a \textd{spatial fragment}.

Two sets $\{ p_{1} \dots p_{n} \} $ and $\{ q_{1} \dots q_{m} \} $ of points on $C$ are \textd{nested}, if all points of one set lie in between two points of the other set. The elements of two nested sets do not conflict by definition. Two fragments $f_a$ and $f_b$ are nested if their endpointsets on $C$ are nested. 

If a basepoint on $C$ is given and a pair of nested fragments, call the fragment \textd{outer fragment} whose endpoints are first reached form the basepoint. Its endpointset is called \textd{outer points}. Call the other fragment \textd{inner fragment} and its endpointset \textd{inner points}. Fix a point $p$ on $C$, take an orientation of $C$ and a parametrisation $f:[0,2\pi]~\rightarrow~C, f(0)=f(2\pi)=p$ that respects the orientation. For two points $a$ and $b$ on $C$ we say that $a<b$ if $f^{-1}(a)<f^{-1}(b)$.

\begin{rem}[Cases of non-conflicting fragments] \label{4cases}
Let $f_a$ and $f_b$ be two fragments of a connected graph with respect to a cycle $C$. Let $v(f_a)=v_{a1},\dots, v_{ar}$  be the endpointset of $f_a$ and let \mbox{$v(f_b)=v_{b1},\dots, v_{br}$} be the endpointset of $f_b$. Then it follows from the definition of conflicting that $f_a$ and $f_b$ do not conflict if and only if $f_a$ or $f_b$ are attached to $C$ in one point only or if for all elements $v(f_a)$ and $v(f_b)$ (up to transposition of $a$ and $b$) one of the following two cases holds: 
\begin{enumerate}
\item $v_{a1} < \dots < v_{ar} \leq v_{b1} < \dots  < v_{br}$ 
\item there exist two points $v_{ai}$ and $v_{a(i+1)}$ in $v(f_a)$  so that $v_{ai}\leq v_{b1},\dots, v_{br}\leq v_{a(i+1)}$ 
\end{enumerate}
The sets $v_{a1},\dots, v_{ar}$ and $v_{b1},\dots, v_{br}$ in the second case above are nested; the inner points are the points $v_{b1},\dots, v_{br}$.
\end{rem}

The \textbf{conflict graph} of a cycle $C$ in a graph $G$ is constructed by introducing a vertex $u_i$ for every fragment~$f_i$ of $G$ with respect to $C$ and adding an edge between the vertices $u_i$ and $u_j$ if and only if the fragments $f_i$ and $f_j$ conflict (Figure~\ref{conflictgraph}). In the proof of Theorem~\ref{goal}, we will use the planarity of~$G$ in form of the following statement:

\begin{figure}
	\def\svgwidth{300pt}
	 \input{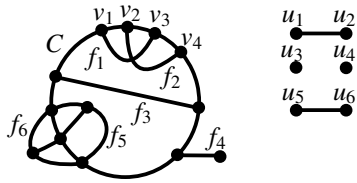}
\caption{The graph shown on the left side is planar. It has a bipartite conflict graph which is shown on the right.}
\label{conflictgraph}
\end{figure}

\begin{thm}[Tutte's Theorem~\cite{Tutte}] \label{Tutte}
A graph~$G$ is planar if and only if the conflict graph of every cycle in $G$ is bipartite. 
\end{thm}
\noindent
A graph is \textbf{bipartite} if its vertices can be divided into two disjoint sets~$S_1$ and~$S_2$ such that every edge of the graph has one endpoint in~$S_1$ and the other endpoint in~$S_2$.

\begin{lem}[Connectivity Lemma~\cite{short}] \label{conn}
The image $\mathcal{G}$ of an embedding $f:G\rightarrow T^{2}\subset \R^3$ of a graph $G$ with $n >1$ connected components on the torus $T^2$  contains either a nonsplit link, or contains no nonsplit link and decomposes into $n$ disjoint components of which at least $n-1$ components are trivial embedded in $R^3$.
\end{lem}

\subsection{Proof}\label{proof}
\begin{proof}{\textit{(Theorem~\ref{goal})}}\\
\begin{enumerate}

\item
\textbf{Reducing the types of spatial graphs:} We show that it is sufficient to consider connected planar spatial graphs that are embedded on the torus~$T^2$ and do neither contain a nontrivial knot nor a nonsplit link, but do contain a torus unknot~$T(1,n), n>0$.
\vskip 1pt
We can assume that the graph $G$ is connected by Lemma~\ref{conn}. Furthermore, we claim that it is sufficient to consider planar spatial graphs $\mathcal{G}$ that contain a trivial torus knot of the form $T(1,n), n>0$ (respectively $T(n,1)$) since if the only knot types contained in $\mathcal{G}$ are $T(0,0)$, $T(0,1)$ and $T(1,0)$, $\mathcal{G}$ is trivial. We see this below by a case study where we restrict the knot types that occur in the spatial graph $\mathcal{G}$. For a knot of knot type $K$ let $\#K$ denote the number of disjoint copies of K.
\begin{enumerate}
\item $\#T(0,0)=n$\\
If the only knot type contained in the spatial graph $\G$ is $T(0,0)$,  there exists a meridian and a longitude of the torus that do not intersect $\G$. Therefore, $\G$ is trivial.
\item $\#T(0,0)=n, \#T(0,1)=k$ (respectively $\#T(0,0)=n, \#T(1,0)=k$)\\
There exists either a meridian or a longitude of the torus that does not intersect $\G$. Therefore, $\G$ is trivial.
\item $\#T(0,0)=n, \#T(0,1)=1, \#T(1,0)=k$\\(respectively $\#T(0,0)=n, \#T(1,0)=1, \#T(0,1)=k$)\\
$\G$ is trivial.\\
Note, that all graphs of this case are essentially of the form drawn left in Figure~\ref{Fig3}. They can at most differ by fragments that are attached to the meridian or a single longitude only (Figure~\ref{Fig3},a). Adding these fragments does not affect the triviality. That the graphs do not look more complicated can be seen by a contradiction: If at least one fragment is added that has endpoints on two longitudes (Figure~\ref{Fig3},b) respectively on the meridian and on a longitude (Figure~\ref{Fig3},c), this introduces a cycle with segments in both the meridian and a longitude such that the cycle bounds a disc in the torus. This is not possible by the assumptions of this case, since the existence of such a cycle ensures the existence of a knot of type $T(1,1)$.
\item $\#T(0,0)=n, \#T(0,1)=k, \#T(1,0)=m$ with $k,m>1$ \\
This case does not fulfil the assumptions since  $\G$ also contains the unknot~$T(1,1)$ (see right of Figure~\ref{Fig3},d).
\end{enumerate}

\begin{figure}[h]
	\centering
	\def\svgwidth{415pt}
	 \input{FIg3c.pdf_tex}
\caption{(a-c):  If $\G$ contains only one copy of $T(0,1)$, or respectively $T(1,0)$, then $\G$ is trivial. (d): If $\G$ contains disjoint copy of both $T(0,1)$ and $T(1,0)$, $\G$ also contains the unknot~$T(1,1)$.}
\label{Fig3}
\end{figure}

\item
\textbf{Existence of a meridian of $T^2$ that intersects $\mathcal{G}$ in only one point:} Beside for some elementary cases that can be investigated directly, we show this with a Morse-theoretical argument that gives a contradiction: If every meridian of $T^2$ would intersect $\mathcal{G}$ in at least two points and if $\mathcal{G}$ would have a subgraph $T(1,n)$, $n>0$, then $\mathcal{G}$ would contain either a nontrivial knot or a nonsplit link.
\vskip 1pt
Note that the case where $\mathcal{G}$ is the union of $T(1,n)$ and a longitude does not fulfil the assumptions since there exists a meridian that intersect $\mathcal{G}$ in one point only (there are $n$ such meridian, namely one for each intersections of $T(1,n)$ with the longitude).
To construct a contradiction, assume that every meridian of $T^2$ intersects $\mathcal{G}$ in at least two points. Cut $T^2$ along a meridian that intersects the graph with minimal number to get a cylinder~$[0,1] \times S^1$. Define the projection function $f:[0,1] \times S^1 \rightarrow [0,1]; \{x, \alpha \} \mapsto x$. Let $S$ be the set of all pairs of pairwise different paths on the cylinder where one path forms a knot of type $T(1,n)$ and the other path $\pi$ has an endpoint in $f^{-1}(0)$ (Figure~\ref{morse}a). The set $S$ is non-empty because there exists at least one cycle $T(1,n)$ by step~(1) and a second path $\pi$ with endpoint in $f^{-1}(0)$ by the assumption that every meridian of $T^2$ intersects $\mathcal{G}$ in at least two points. Note that $\pi$ either intersects $T(1,n)$ or has the second endpoint on $f_{1}(1)$ otherwise, since the number of intersections between $f^{-1}(0) \cap \mathcal{G}$ is minimal by assumption. Now take a pair of paths in $S$ which maximises the value $f(t_{n})$ where $t_n$ is the intersection point of the two paths in the pair.  

\begin{figure}[h]
	\centering
	\def\svgwidth{420pt}
	 \input{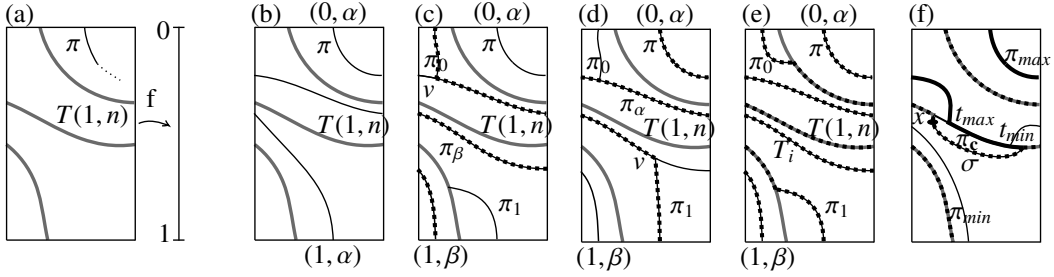}
\caption{Illustrating step~2 of the proof.}
\label{morse}
\end{figure}
\vskip 1pt
If $T(1,n)$ and $\pi$ do not intersect, then there exist two disjoint paths on the cylinder. In this situation, there exists a nontrivial knot or a nonsplit link since the graph is connected as we show now:\\
If $\pi$ is a cycle, a nonsplit link is formed by $T(1,n)$ and $\pi$ (Figure~\ref{morse}b). So let $\pi$ be a path with endpoints $(0, \alpha) \in f^{-1}(0)$ and $(1, \beta) \in f^{-1}(1), \alpha \neq \beta$. As the graph intersects $f^{-1}(0)$ and $f^{-1}(1)$ in a minimal number, wlog we can assume that there exist paths~$\pi_{0}$ from $(0, \beta)$ and $\pi_{1}$ from $(1, \alpha)$ to $T(1,n)$ or to $\pi$ (since the graph is connected).
\begin{enumerate}
\item 
If $\pi_{0}$ intersects $\pi$ in a point $v$ before intersecting $T(1,n)$ (Figure~\ref{morse}c), denote the segment of $\pi$ between $v$ and $(1, \beta)$ by $\pi_{\beta}$. Such a graph contains a nonsplit link where one component is $T(1,n)$ and the other consists of $\pi_{0}$ and $\pi_{\beta}$.\\
If $\pi_{1}$ intersects $\pi$ in a point $v$ before intersecting $T(1,n)$  (Figure~\ref{morse}d), denote the segment of $\pi$ between $(0, \alpha)$ and $v$ by $\pi_{\alpha}$. Again, such a graph contains a nonsplit link where one component is $T(1,n)$ and the other consists of $\pi_{\alpha}$ and $\pi_{1}$. Up to this point the arguments applies to $k=T(1,1)$ as well.
\item
If both paths~$\pi_{0}$ and $\pi_{1}$ intersect $T(1,n)$ before they intersect $\pi$ (Figure~\ref{morse}e), denote the segment (possibly a point) of $T(1,n)$ that lies between the endpoints of $\pi_{0}$ and $\pi_{1}$ by $T_{i}$. The cycle that runs through $\pi$, $\pi_{0}$, $T_i$ and $\pi_{1}$ is a nontrivial knot for $n>1$.\\
If $T(1,n)=T(1,1)$, there are possibilities to connect $(0, \beta)$ and $(1, \alpha)$ to $T(1,1)$ without introducing a nontrivial knot or a nonsplit link. But considering $T(1,1)$, we could exchange the meridian and the longitude in the argument. This gives the extra condition that not only each meridian but also each longitude intersects the graph in at least two points. An elementary investigation shows directly that Theorem~\ref{goal} is valid in those cases.
\end{enumerate}
If $T(1,n)$ and $\pi$ do intersect, there exists a point of maximal intersection $t_{max}$, $0<f(t_{max}) < 1$. Now, consider the set of all paths that are different from $T(1,n)$ and that have one endpoint on $f^{-1}(1)$. Take one path of this set which minimises the value of $f(t_{min})$ where $t_{min}$ is the point of intersection between the path and $T(1,n)$ and call that path $\pi_{min}$. If $f(t_{min}) < f(t_{max})$, the argument is very similar to the case above.
\vskip 1pt
So let us finally consider the case where $f(t_{max}) \leq f(t_{min})$ (Figure~\ref{morse}f). Denote the component of $\mathcal{G}-T(1,n)$ that contains $\pi_{min}$ ($\pi_{max}$) by $c_{min}$ ($c_{max}$). As before, since the graph intersects every meridian at least twice by assumption, there exists a point~$x \nsubset \pi$ in the graph so that $f(x) = f(t_{max})$. No path disjoint from $T(1,n)$ containing $x$ can connect to $c_{min}$ as this would contradict the maximality of $t_{max}$. Similarly, no path disjoint from $T(1,n)$ containing $x$ can connect to $c_{max}$ as this would contradict the minimality of $t_{min}$. Therefore, every path through $x$ connects to $T(1,n)$ before and after $t_{max}$ (and similarly before and after $t_{min}$) and is disjoint from $c_{max}$ and $c_{min}$. Take such a path and denote it by $\sigma$. Replacing the segment $T_{i}$ of $T(1,n)$ that runs between the endpoints of $\sigma$ with $\sigma$ gives us a torus knot $T'(1,n)$ (dotted in Figure~\ref{morse}f). Since $T'(1,n)$ is disjoint from $\pi_{max}$, there exists a path $\pi_{c}$ (fattened in Figure~\ref{morse}f) consisting of $\pi$ and $T_{i}$ that is disjoint from $T'(1,n)$ until some time after $t_{max}$. This contradicts the maximality of the pair we selected from $S$.
\vskip 7pt

\item 
\textbf{Deforming $T(1,n)$ to a longitude and finding a diagram~$D_{R'}$ of $\mathcal{G}$:}
\vskip 1pt
Recall that the spatial graph $\mathcal{G}$ contains no nontrivial knots or links by assumption. By putting $\mathcal{G}$ into a general position, there exists a meridian~$m$ of the torus that intersects $\mathcal{G}$ in exactly one point~$P$. The existence of $P$ is given by the previous step~(2). The point~$P$ lies on $T(1,n)$ since every meridian intersects $T(1,n)$. Now perform the following twist: By cutting the torus~$T$ along the meridian~$m$, then twisting it $n$-times around the core of $T$ and identifying the same points again afterwards, an ambient isotopy $i: \mathcal{G} \rightarrow \mathcal{G'}$ of the spatial graph is induced that maps $T(1,n)$ onto the longitude $l=T(1,0)$ of a new torus~$T'$ (not isotopic to $T$). We denote the image $i(\mathcal{G})$ on $T'$ by $\mathcal{G}'$. Restricted to the meridian~$m$ of $T$, the isotopy is the identity by construction. Therefore, $\mathcal{G}'$ and $m$ intersect in $l$ only. Define $Z' := T' \setminus (m \setminus P)$. 
\vskip 1pt
\noindent
Let us furthermore consider the diagram $D_{R'}$ of $\mathcal{G}'$ that we obtain the following way: We project $Z'$ onto a half-open rectangle $R'=(0,1) \times [0,1] \cup \{ (0,0) \} \cup \{ (1,0) \}$, where $P$ is projected on both corner points $\{ (0,0) \}$ and $\{ (1,0) \}$, such that $l$ is the bottom line of~$R'$ and take a generic position so that the top line of~$R'$ does not intersect the graph in vertices. As usual, we indicate the over- and under-crossings of $\mathcal{G}'$ (Figure~\ref{rectangleprojection}). Wlog we can assume that the diagram is regular, i.e., the diagram has only finitely many multiple points which all are transversal double points and no vertex is mapped onto a double point. Furthermore, let $D_{R'}$ be a reduced diagram, i.e., a diagram with the minimal number of crossings that can be achieved from projecting  $Z'$ onto a rectangle $R'$ as described. 
\begin{figure}[h]
	\centering
	\def\svgwidth{420pt}
	 \input{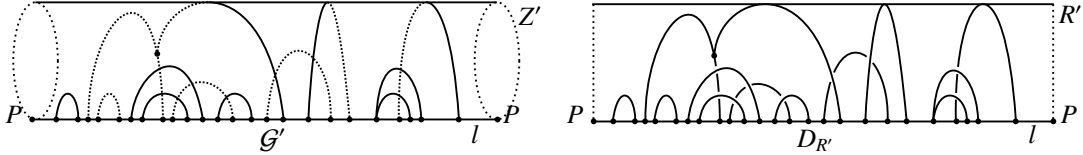}
\caption{The diagram $D_{R'}$ of $\G'$ obtained by projecting $Z'$ onto $R'$.}
\label{rectangleprojection}
\end{figure}


\item
\textbf{Showing that pairs of spatial fragments in a reduced diagram~$D_{R'}$ have no crossings if they are non-conflicting and only one type of crossings if they are conflicting:}
\vskip 1pt 
By Tutte's Theorem~\ref{Tutte}, any cycle of a planar graph has a bipartite conflict graph. As $G$ is planar by assumption, it follows that the conflict graph of $l$ in $G$ is bipartite. As the graph~$G$ is connected, all fragments of $G$ with respect to $l$ fall into two sets~$S_1$ and $S_2$ so that fragments which are elements of the same set do not conflict.
Choose an orientation of $l$. Starting at the point~$P$, enumerate along the orientation all vertices $v_1, \dots, v_k$ of $l$ that are endpoints of fragments of $G$ with respect to $l$. ($P$  might or might not be the element of a fragment's endpointset.) Denote the spatial fragments of $\mathcal{G}'$ by $f_{1}, \dots , f_{n}$ respecting the orientation of $l$ and so that $v_1 \in v(f_{1})$ . Assign to each fragment~$f_i$ of $G$ with respect to $l$ its endpointset $ v(f_i)=v_{l_i} \le \dots \le v_{r_i} \subset \{v_1, \dots, v_k\}$. We show that each pair of spatial fragments $f_i$ and $f_j$ has either no crossings or wlog  $f_i$ over-crosses $f_j$ at every crossing of the diagram $D_{R'}$:

\begin{enumerate}
\item
This is clear if $i=j$ since a fragment over-crosses (as well as under-crosses) itself at every self-crossing.

\item
If $i \neq j$ and $f_i$ and $f_j$ do not conflict, they have no crossings in a reduced diagram $D_{R'}$ since one of the cases in Remark~\ref{4cases} holds: It is clear that there are no crossings between $f_i$ and $f_j$ in the first case of Remark~\ref{4cases}. In the second case, let wlog $f_j$ be the inner fragment.  Then, $f_j$ lies entirely inside the cell of $R' \setminus f_{i}$ that has $[ v_{l_j}, v_{r_j} ]$ as part of the boundary. It follows from the connectivity of fragments that $f_i$ and $f_j$ have no crossings in a reduced diagram $D_{R'}$.

\item
If $i \neq j$ and $f_i$ and $f_j$ are conflicting, only crossings of one type can occur. Since the entire spatial graph $\mathcal{G}'$ is an embedding in $Z'$ (as well as in $T'$), it is not possible that both crossing types between $f_i$ and $f_j$ occur (Figure~\ref {disc}).
\vskip 2pt
We remark that if $i \neq j$ and $f_i$ and $f_j$ are conflicting, they have at least one crossing in $D_{R'}$: Without affecting $D_{R'}$, a fragment $\bar{f}$ can be added in $(0,1) \times [-1,0] \cup \{ (0,0) \} \cup \{ (1,0) \}$ such that $\bar{f}$ conflicts with both $f_i$ and $f_j$. Then, the subgraph consisting of $l$, $f_i$, $f_j$, and $\bar{f}$ is nonplanar by Tutte's Theorem~\ref{Tutte}. The fragment $\bar{f}$ does not lie in $D_{R'}$ by construction. Therefore, there exist neither crossings between $\bar{f}$ and $f_i$ nor between $\bar{f}$ and $f_j$. It follows from the planarity that a crossing between $f_i$ and $f_j$ must exists.

\end{enumerate}
\begin{figure}[h]
\begin{center}
		\includegraphics[scale=0.6]{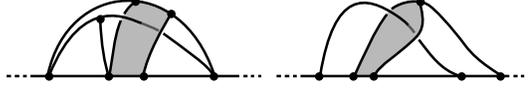}
\caption{A pair of conflicting spatial fragments in a reduced diagram $D_{R'}$ with both over- and under-crossings does not embed on the torus.}
\label{disc}
\end{center}
\end{figure}

\item
\textbf{Showing that a spatial fragment that conflicts with a pair of nested spatial fragments has the same crossing type with both of them:}
\vskip 1pt
If a fragment $f_i$ conflicts with two fragments $f_{j1}$ and $f_{j2}$, it follows from the bipartiteness of the conflict graph that $f_{j1}$ and $f_{j2}$ do not conflict. If $f_{j1}$ and $f_{j2}$ satisfy case~(1) of Remark~\ref{4cases}, it cannot be concluded whether $f_i$ over- or under-crosses $f_{j2}$ from knowing that $f_i$ over- or under-crosses $f_{j1}$. However, if $f_{j1}$ and $f_{j2}$ are nested as in the second case of Remark~\ref{4cases} with $f_{j1}$ being the inner fragment, and if $f_i$ wlog over-crosses $f_{j1}$ in $D_{R'}$, then $f_i$ also over-crosses $f_{j2}$. We see this with a contradiction (Figure~\ref{onetype}):  As in~(c) of step~(4) above, there exists an element $v_{ia} \in v(f_i)$ such that $v_{l_{j1}} < v_{ai}< v_{r_{j1}}$. Assume that $f_i$ over-crosses $f_{j1}$ in a non-empty set of points but under-crosses $f_{j2}$ in a non-empty set in $D_{R'}$. By the connectivity of fragments, there exists a path $p:[0,1] \rightarrow f_{i} \cup v_{ia} \cup v_{ib}$, with $v_{ib} \in v(f_i), a \neq b$ with endpoints $p(0)=v_{ia}$ and $p(1)=v_{ib}$, that over-crosses $f_{j1}$ in $p(t_{1})$ and under-crosses $f_{j2}$ in  $p(t_{2})$. To change from an over- to an under-crossing, a path in $D_{R'}$ has to intersect either the bottom or the top line of $R'$. As $p$ does only intersect the bottom line in $v_{ai}$ and $v_{ib}$, $p$ must have an intersection point $p(t)$ with the top line of $R'$, so that $t_{1} < t < t_{2}$. But as $f_{j1}$ is nested in $f_{j2}$, it follows that $v_{l_{j2}} < v_{ia}< v_{r_{j2}}$. Therefore, $p$ starting from $v_{ia}$ over-crosses $f_{j1}$ as well as under-crosses $f_{j2}$ before it can intersect the top line of $R'$. This is a contradiction.
\begin{figure}[h]
	\centering
	\def\svgwidth{150pt}
	 \input{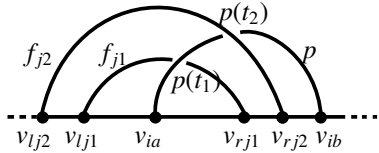}
\caption{It is only possible to have different crossing types between a fragment and two nested fragments with whom the first fragment conflicts if the graph is not realised on the torus.}
\label{onetype}
\end{figure}

\item
\textbf{Separating conflicting fragments to get a reduced diagram~$D_{R''}$ in which no pair of fragments has crossings:}
\vskip 1pt
The conclusions made in~(3) and~(4) allow rotations of the spatial fragments of $\mathcal{G}'$ around the longitude~$l$ in $\R^3$ which gives an ambient isotopy from $\mathcal{G}'$ to a realisation $\mathcal{G''}$ in which all fragments that are elements of $S_1$ lie on the torus~$T'$ and all elements of $S_2$ lie on a second torus~$\hat{T}$. The torus~$\hat{T}$ is glued to $T'$ in $l$ along a longitude.
\vskip 1pt
The rotations can be chosen as follows (compare Figure~\ref{rotation}): 
 Let $\mathcal{F}^{1}$ be the set of all fragments of $\G'$ with respect to $l$ (Figure~\ref{rotation}a). Order the fragments according to their first endpoint on $l$. If two fragments have identical first endpoints, the last endpoints of the fragments are compared and the fragment with bigger last endpoint is counted first. If the last endpoints coincide as well, the outermost fragment is counted first. If both fragments have only two endpoints, the order can be chosen arbitrarily.  If a step during the procedure which is described below cannot be performed, continue with the next step. (Figure~\ref{rotation},b shows the diagrammatic description of the starting situation.) 
\begin{enumerate}
\item
\begin{enumerate}
\item
Let $f_{k_{11}}$ be the first fragment that conflicts with a fragment $f_{i}, i<k_{11}$. Define $\mathcal{F}_{k_{11}}$ iteratively, starting with $\mathcal{F}_{k_{11}}=f_{k_{11}}$ as the set of all fragments that are nested or are in conflict with a fragment in $\mathcal{F}_{k_{11}}$.
Then rotate all spatial fragments that are elements of the set $\mathcal{F}_{k_{11}}$ rigidly in $\R^3$ by $\pi$ around the longitude~$l$. They are now embedded on $\hat{T}$. It is possible to choose the direction of the rotation so that no spatial fragments pass through each other since the spatial fragments in $\mathcal{F}_{k_{11}}$ have only over-crossings (respectively only under-crossings) with fragments $f_{i}, i<k_{1}$ by construction and by steps~(3) and~(4)  (Figure~\ref{rotation}c). 
\begin{enumerate}
\item
Let $f_{k_{12}}$ be the first fragment that is not an element of $\mathcal{F}_{k_{11}}$ but conflicts with a fragment $f_{i}, i<k_{11}$. Define $\mathcal{F}_{k_{12}}$ analogously to $\mathcal{F}_{k_{11}}$ (Figure~\ref{rotation}c). Since the elements of $\mathcal{F}_{k_{12}}$ are neither nested nor conflicting with any elements of $\mathcal{F}_{k_{11}}$, by the same argument as above, there is a rigid rotation of $\mathcal{F}_{k_{12}}$ in $\R^3$ by $\pi$ around the longitude~$l$ that does not pass the spatial graph through itself (Figure~\ref{rotation}d). 
\item
Continue this procedure for all remaining spatial fragments that conflict with a fragment $f_{i}, i<k_{11}$. Let $\mathcal{F}_{1} := \mathcal{F}_{k_{11}} \cup \mathcal{F}_{k_{12}} \cup \dots $. Then, $\mathcal{F}_{1}$ is embedded on $\hat{T}$ (Figure~\ref{rotation}d).
\end{enumerate}
\item
Let $f_{k_{21}} \notin \mathcal{F}_{1}$ be the first fragment that conflicts with a fragment $f_{i}, \; i < k_{21}$  ($i>k_{11}$ by construction). Define $\mathcal{F}_{k_{12}}$ analogously to the previous steps (Figure~\ref{rotation}c) and perform the rotation around $l$ (Figure~\ref{rotation}d). Continue this procedure for all remaining spatial fragments that conflict with a fragment $f_{i}, i<k_{21}$. Let $\mathcal{F}_{2}:=\mathcal{F}_{k_{21}} \cup \mathcal{F}_{k_{22}} \cup \dots  $. 
\item
Continue with this procedure until all fragments $f_{1} \dots f_{n}$ have been considered. The fragments that are elements of the set $\mathcal{F}^2 := \mathcal{F}_{1} \cup \mathcal{F}_{2} \cup \dots $ are now embedded on $\hat{T}$ (Figure~\ref{rotation}d).
\end{enumerate}
\item
Start (a) again beginning with the subgraph of $\G'$ that corresponds to $\mathcal{F}^2$. Note that during this step the rotations bring fragments back onto the torus~$T'$ but will not introduce crossings with $\G' - \mathcal{F}^2$ (Figure~\ref{rotation}e). 
\item
Continue the procedure has to be continued until all elements of $S_1$ lie on the torus~$T'$ and all elements of $S_2$ lie on the torus~$\hat{T}$ (Figure~\ref{rotation}f).
\end{enumerate}
\enlargethispage \baselineskip
This gives a realisation~$\G''$ of $G$ which is ambient isotopic to~$\G$. By~(3), a pair of spatial fragments of~$\G''$ has no crossings in a reduced diagram~$D_{R''}$ of~$\G''$ on a rectangle $R''=((0,1) \times [-1,1]) \cup \{ (0,0) \} \cup \{ (1,0) \} = ((l\setminus P) \times [-1,1]) \cup \{ P \} $. The diagram $D_{R''}$ is the composition of two diagrams defined as in~(2) for $T'$ on $((l\setminus P) \times [0,1]) \cup \{ P \} )$ and analogously for $\hat{T}$ on $((l\setminus P) \times [-1,0]) \cup \{ P \} $ (Figure~\ref{rotation}.g). 

\begin{figure}[h]
	\centering
	\def\svgwidth{460pt}
	 \input{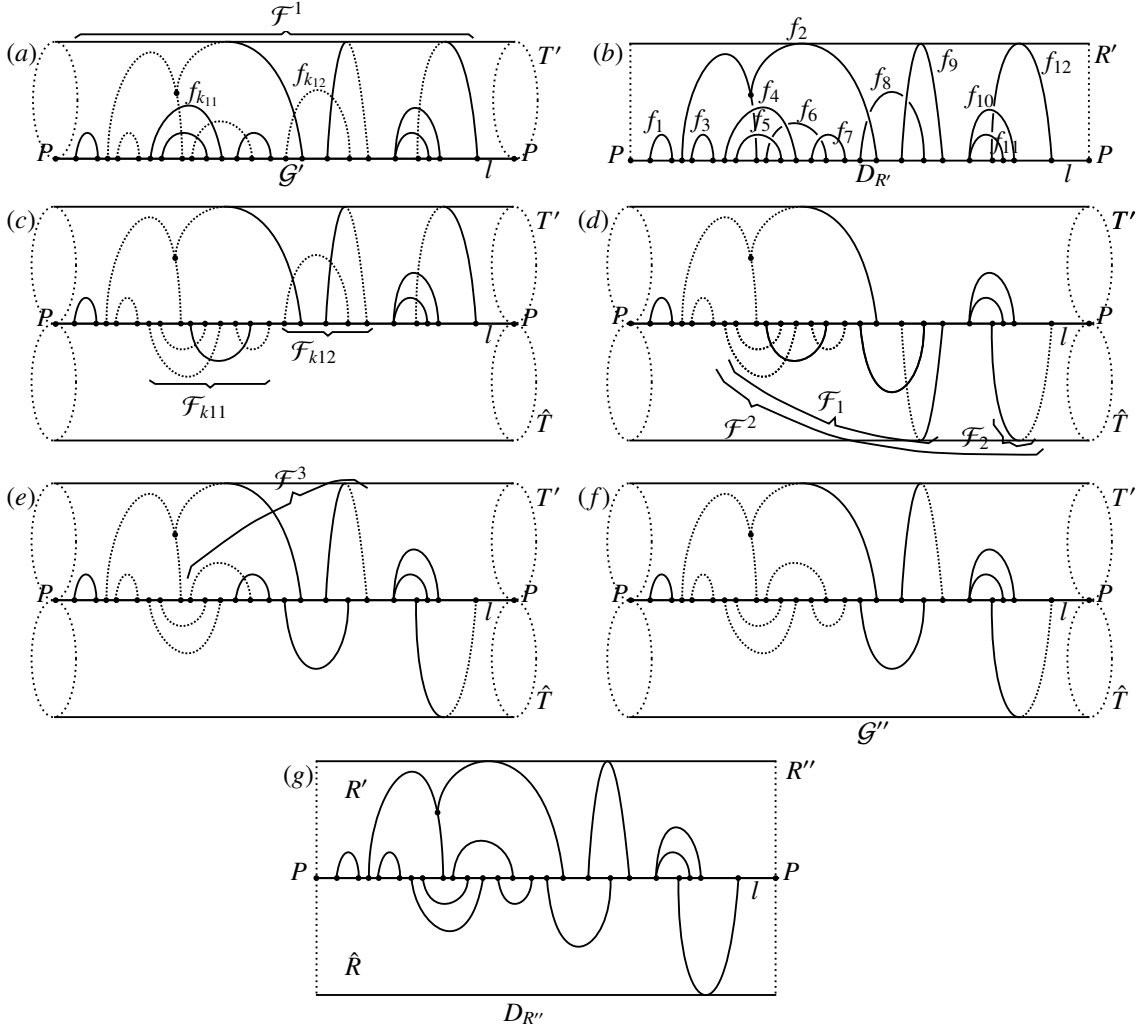}
\caption{The rotation of fragments described in~(6). For clarity of the figure, each fragment is chosen not to have crossings with itself.}
\label{rotation}
\end{figure}


\vskip 7pt
\item
\textbf{Showing that a single spatial fragment $f_i$ has no crossings in a reduced diagram~$D_{R''}$ of~$\mathcal{G}''$ (Fig.~\ref{1fragment}):}
\vskip 1pt

Each spatial subgraph $f_{i} \cup [v_{li}, v_{ri}]$ is embedded on a sphere~$S_{i}^{2}$. To see this, let wlog $f_i$ be embedded on $T'$ and take two meridians of $T'$ intersecting $l$ in $v_{li}$ and $v_{ri}$ such that the meridians do not intersect $\mathcal{G}''$ except in $v_{li}$ and $v_{ri}$. Then glue two meridional discs in, one in each meridian. $S_{i}^{2}$ consists of the two meridional discs and the part of $T'$ where $f_i$ is embedded in that lies between the meridians. We now want to ambient isotope $f_i$ inside the ball bounded by $S_{i}^{2}$ where we take the inside to be the component of $\R^3 \setminus S_{i}^{2}$ that does not intersect $\mathcal{G}''$. This isotopy will transform the diagram~$D_{R''}$ to a diagram in which the subdiagram corresponding to $f_i$ is crossing free.\\
Take  the subdiagram~$D^{i}_{R''}$ of the diagram~$D_{R''}$ that corresponds to $f_{i} \cup [v_{li}, v_{ri}]$. Perform all reducing Reidemeister~I \& II moves on it (Figure~\ref{1fragment}, a-b). Simplify $D^{i}_{R''}$ by isotopy whenever possible. We can assume that the diagram~$D^{i}_{R''}$ has edges crossing the top line of $R''$; otherwise it is crossing free and we are done. If one of those edges intersects the top line of $R''$ more than once or runs through a crossing in $D^{i}_{R''}$, we subdivide the edge by adding vertices (Figure~\ref{1fragment}, a-b). Therefore, we can assume that an edge of $D^{i}_{R''}$ that intersects the top line of~$R''$ intersects it only once and is crossing free in the diagram. (This does not affect our argument since if a subdivision of a spatial graph is trivial, the spatial graph itself is.) Denote the edges crossing the top line of $R''$ by  $\{ e_1 , \dots , e_{\tilde{k}} \}$. Each edge $e_j$ has two endpoints, $e_{ju}$ and $e_{jo}$ (fat in Figure~\ref{1fragment}b). By the connectivity of $\mathcal{G''}$, there is at least one edge~$e_j$ in $\{ e_1 , \dots , e_{\tilde{k}} \}$ for which there exists a path in $f_i$ from an endpoint of the edge wlog $e_{ju}$ to an element $v \in v(f_{i})$ that does not intersect the top line of $R''$. Denote such a path with endpoints $e_{ju}$ and $v$ by $p(e_{ju},v)$ (fat in Figure~\ref{1fragment}b). The set of all such paths is called $\{p(ju) \}$. The set $\{p(jo) \}$ is analogously defined for the endpoint~$e_{jo}$ of $e_j$.\\
If an edge $e_j \in \{ e_1 , \dots , e_{\tilde{k}} \}$ has an endpoint $e_{ju}$ or $e_{jo}$ which is not the endpoint of any path in $\{p(ju)\}$ or $\{p(jo)\}$ ($e_{3u}$ in Figure~\ref{1fragment}b), $e_j$ can be deformed not to intersect the top line of $R''$ by moving $e_j$ away from the top line while keeping $D_{R''} - e_{j}$ fixed. After this procedure, a subset of edges $\{ e_1 , \dots , e_{k}\} \subseteq \{ e_1 , \dots , e_{\tilde{k}} \}$ remains in which every edge~$e_j$ has endpoints $e_{ju}$ and $e_{jo}$ such that there exist at least two (possibly constant) paths $p(e_{ju},v_a)$ and $p(e_{jo},v_b)$ with $v_{a}, v_{b} \in v(f_i)$. Such a path $p(e_{ju},v_a)$ or $p(e_{jo},v_b)$ cannot have both over-crossings and under-crossings since in this case the path would intersect $\partial R''$ which it does not by construction. Also by construction, if a path in $\{p(jx)\}, x=u,o$ has an over-crossing (or respectively under-crossing), no path in $\{p(jx)\}$ has an under-crossing (respectively over-crossing). In addition, we can assume wlog  that there is no edge~$e_j \in \{ e_1 , \dots , e_{k}\}$ that has an endpoint $e_{ju}$ or $e_{jo}$ so that all paths $\{p(ju)\}$ or $\{p(jo)\}$ are crossing free ($e_{\tilde{k}u}$ in Figure~\ref{1fragment}b) as in this case we can deform $e_j$ away from the top line while keeping $D_{R''} - e_{j}$ fixed. Also, if $\{p(ju)\}$ and $\{p(jo)\}$ have only one type of crossings ($e_{\tilde{k}-2}$ in Figure~\ref{1fragment}b), we can again deform $e_j$ away from the top line while keeping $D_{R''} - e_{j}$ fixed. Therefore, every edge~$e_j$ has one endpoint $e_{ju}$ such that at least one path in $\{p(ju)\}$ has crossings which all are under-crossings and one endpoint $e_{jo}$ such that at least one path in $\{p(jo)\}$ has crossings which all are over-crossings. Furthermore, in $\{p(ju)\}$ (respectively $\{p(jo)\}$) are no paths that have over-crossings (respectively under-crossings) by definition of the paths.  \\
Assign to each edge $e_j \in \{ e_1 , \dots , e_{k} \}$ the set $w(e_{ju}) \subseteq \{ w_{j1}, \dots , w_{jl} \}$ (analogously $w(e_{jo}) \subseteq \{ w_{j1}, \dots , w_{jl} \}$) which is the set of points in $v(f_i)$ that are endpoints of at least one element in $\{p(ju)\}$ (respectively $\{p(jo)\}$). The union of the two sets is the endpointset of $e_j$ denoted by $w(e_j)=\{ w_{j1}, \dots , w_{jl} \} := w(e_{ju}) \cup w(e_{jo})$. The sets in the example in Figure~\ref{1fragment}c are $w(e_{ku}) =\{ w_{7}, w_{8}, w_{9} \}$ and $w(e_{ko})=\{ w_{1}, w_{2}, w_{7}, w_{10}, w_{11}, w_{12}  \}.$ Denote the union of $w(e_{1u}) \cup \dots \cup w(e_{ku})$ by $w(u)$ (and the union $w(e_{1o}) \cup \dots \cup w(e_{ko})$ by $w(o)$). In Figure~\ref{1fragment}c, these are $w(u)=\{w_{3}, w_{4}, w_{5}, w_{6}, w_{7}, w_{8}, w_{9}  \}$ and $w(o)= \{w_{1}, w_{2}, w_{7}, w_{10}, w_{11}, w_{12} \}$.\\
See that there exist no four points $w_{u1}, w_{u2} \in w(e_{ju})$ and $w_{o1}, w_{o2} \in w(e_{jo})$ that are interlaced as wlog  $w_{u1} < w_{o1} < w_{u2} < w_{o2}$: the cycle $(l - (w_{u1}, w_{o2})$, $p(e_{ju},w_{u1})$, $e_{j}, p(e_{jo},w_{o1})$, $[w_{o1}, w_{u2}]$, $[w_{u2}, w_{o2}])$ would have three pairwise conflicting fragments $[w_{u1}, w_{o1}]$, $p(e_{ju},w_{u2})$ and $p(e_{jo},w_{o2})$ (Figure~\ref{edgeends}). This contradicts the bipartiteness of $\mathcal{G}''$ which by Theorem~\ref{Tutte} contradicts its planarity. (Also, this graph forms $K_{3,3}$ where the points $w_{u1}, e_{jo}, w_{u2}, w_{o1}, e_{ju}, w_{o2}$ are the vertices.) Therefore, $w(e_{ju})$ and $w(e_{jo})$ can only be arranged like the endpointsets in Remark~\ref{4cases} and it is allowed to restrict to those cases as done below. By the connectivity of fragments, each element of $v(f_i)$ belongs to $w(u)$, $w(o)$ or to both. This gives a division of $[v_{li}, v_{ri}]$ into intervals, where a new interval starts at each point of $v(f_i)$ that is an element of both $w(u)$ and $w(o)$ or where a new interval starts in a point $v_{k+1}$ if $w(o) \notni v_{k} \in w(u)$ and  $w(u) \notni v_{k+1} \in w(o)$ (or if $u$ and $o$ exchanged). (In Figure~\ref{1fragment}d, the intervals are $[w_{1}, w_{3}], [w_{3}, w_{7}], [w_{7}, w_{10}], [w_{10}, w_{12}]$.)\\
\begin{figure}[h]
	\centering
	\def\svgwidth{150pt}
	 \input{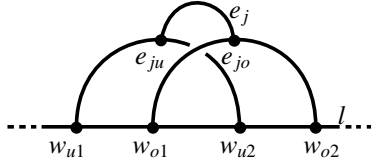}
\caption{The situation where points in which $\{p(ju)\}$ and $\{p(jo)\}$ are attached to the circle~$l$ are interlaced as drawn in the figure cannot occur.}
\label{edgeends}
\end{figure}
\vskip 1pt
The setting is now sufficiently well understood to eliminate all remaining crossings in two cases:
\vskip 1pt
\textbf{Case 1:} $w(e_{ju})$ and $w(e_{jo})$ are nested.\\
Assume that for an edge~$e_j$, $w(e_{ju})$ and $w(e_{jo})$ are nested with wlog  $w(e_{ju})$ being the inner points.
The inner points are all contained in one interval of the bottom line division since if they laid in two intervals, there would exist a point $w(e_{jo}) \notni w_{m} \in w(o)$ between two points of $w(e_{ju})$. Consequently, there would exist a path $p(e_{j'o},w_{m}), j \neq j'$ with endpoints $e_{j'o}$ and $w_{m}$ which does not intersect any of the paths that are elements of $w(e_{jo})$. This is not possible as $R'' \setminus (p(e_{j'o},w_{m}) \cup e_{j'})$ consists of two components of which both contain elements of $w(e_{ju})$ and there exists a path between any point of $w(e_{ju})$ and $e_{ju}$ by definition.\\
Define $\{e^{j}\}$ as the subset of edges $\{e^{j}\} \subseteq \{ e_1 , \dots , e_{k} \}$ so that all edges in $\{e^{j}\}$ have an endpoint in the interval $I$ of the bottom line that contains points of $w(e_{ju})$ (Figure~\ref{1fragment}d: $e^{1}=e^{2}=e^{3}=e^{4} \neq e^{5}=e^{6}$). Now consider the diagram $D_{R''} - \{e^{j}\}$ in which all edges that are elements of the set $\{e^{j}\}$ are deleted (Figure~\ref{1fragment}e). There exists a path $p^{j}:[0,1] \rightarrow D_{R''} - \{e^{j}\}$ from $p^{j}(0) \in w(e_{jo})$ to $p(1) \in w(e_{jo})$ such that $p^{j}(0) \leq w(e_{ju}) \leq p^{j}(1)$ and so that there exist two distinct points $p^{j}(t_{1}), p^{j}(t_{2})$ with $t_{1}, t_{2} \in [0,1]$ that have the following property: The diagram~$D_{R''} - \{e^{j}\} - p^{j}(t_{1}) - p^{j}(t_{2})$ splits such that the component~$C_{jo}$ containing $e_{jo}$ does not contain any points of $[v_{li}, v_{ri}]$ (Figure~\ref{1fragment}e). Furthermore, after performing a Whitney~2-flip on $C_{jo}$ (Figure~\ref{1fragment}, e-f), the edges of $\{e^{j}\}$ can be reintroduced to the diagram~$D_{R''} - \{e^{j}\}$ without introducing crossings (Figure~\ref{1fragment}g). A Whitney~2-flip replaces a component by its mirror image as shown in Figure~\ref{Whitneyflip}, left. This corresponds to a rotation in $\R^3$ by $\pi$ that would not pass the spatial graph~$\mathcal{G}''$ through itself - even if all edges~$\{e^{j}\}$ are left attached (Figure~\ref{Whitneyflip}, right and Figure~\ref{1fragment}, d-g). Therefore, we have an ambient isotopy that eliminates the crossings of $\{e^{j}\}$.\\
After continuing this procedure, all remaining edges of~$\{ e_1 , \dots , e_{k} \}$ have endpointsets so that for each edge all elements of $w(e_{ju})$ are smaller or equal than all elements of $w(e_{jo})$ (or all elements of $w(e_{ju})$ are greater or equal than all elements of $w(e_{jo})$).\\

\begin{figure}
\centering
		\includegraphics[scale=0.7]{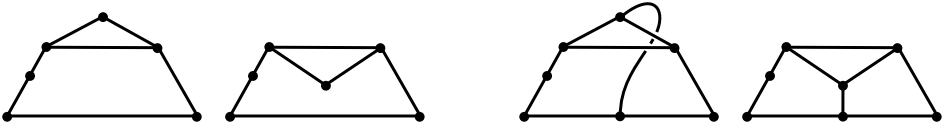}
\caption{Whitney 2-flip and corresponding rotation by $\pi$}
\label{Whitneyflip}
\end{figure}

See with a contradiction that it is always possible to find two points $p(t_{1})$ and  $p(t_{2})$ with the required property as follows. Assume that no two points $p(t_{1})$ and  $p(t_{2})$ with the required property exist. Then there exists a subgraph of $\G''$ that is nonplanar (Figure~\ref{pointexistence}): Wlog there exist two points $w(e_{ju}) \ni w_{1} < w_{2} \in w(e_{jo})$ such that there exists an edge $e_{w2}$ between $w_2$ and an inner point $p_{w2}$ of a path $p(e_{jo},w)$ with $w < w_{1}$. Choose $t_{s}$ as $ 0<t_{s}<1$ and such that $p(e_{jo},w)(t_{s})=p_{w2}$. This allows the description of a cycle with non-bipartite conflict graph (fat in Figure~\ref{pointexistence}), alternatively $K_{3,3}$ is detectable. The cycle with non-bipartite conflict graph consists of the following segments: $l-(w,w_{rj})$, $p|_{[0,s]}$, $e_{w2}$, $[w_{1},w_{2}]$, $p(w_{1},e_{ju})$, $e_{j}$, $p(e_{jo}, w_{rj})$ and the fragments are $(w,w_{1})$, $p|_{(s,1)}$ and $(w_{2}, w_{rj})$. This contradicts planarity by Theorem~\ref{Tutte}.

\begin{figure}[h]
	\centering
	\def\svgwidth{150pt}
	 \input{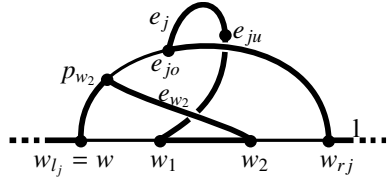}
\caption{Contradiction showing that $p(t_{1})$ and $p(t_{2})$ do exist.}
\label{pointexistence}
\end{figure}
\vskip 1pt

\textbf{Case 2:} All endpoints in $w(e_{ju})$ are smaller or equal to all endpoints in $w(e_{jo})$ (or respectively $w(e_{ju}) \geq w(e_{jo})$).\\
If $\{ e_1 , \dots , e_{k} \}$ is empty or has one element only, it follows that the diagram of $f_i$ has no crossings. So consider the case that  $k \geq 2$. Wlog, assume that all elements of $w(e_{ju})$ are smaller or equal to all elements of $w(e_{jo})$. If all elements of $w(e_{(j+1)u})$ are greater or equal to all elements of $w(e_{(j+1)o})$, it follows that $w(e_{jo}) = w(e_{(j+1)o})$ by construction and the connectivity of fragments (Figure~\ref{notnested}, left).\\
If all elements of $w(e_{(j+1)u})$ are smaller or equal to all elements of $w(e_{(j+1)o})$ but there exists an element~$w_{II} \in w(e_{(j+1)u})$ that is smaller than an element $w_{III} \in w(e_{jo})$, there are four elements $w_{I} < w_{II} < w_{III} < w_{IV}, w_{I} \in w(e_{ju}) , w_{IV} \in w(e_{(j+1)o})$. The paths $p(w_{I}, e_{ju}), e_{j}, p(w_{jo}, e_{III})$ and $p(w_{II}, e_{(j+1)u})$, $e_{j+1},$ $p(w_{jo}$ $, e_{IV})$ are connected via a path in $\mathring{R}''$ by the connectivity of fragments. The subgraph of $\mathcal{G}''$ (fat in Figure~\ref{notnested}, right) consisting of those three paths and $l$ is nonplanar (it is $K_{3,3}$) which can again be shown by an argument similar to the one given above in case~1 by choosing any Hamilton cycle of the subgraph (i.e., a cycle that runs through every vertex of the subgraph once) and seeing that its conflict graph is not bipartite. Therefore, it is shown in step~(7) that a reduced form of the diagram~$D_{R}''$ has no crossings. 
\begin{figure}[h]
	\centering
	\def\svgwidth{380pt}
	 \input{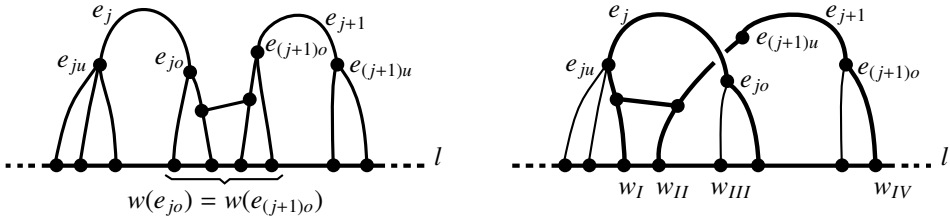}
\caption{Case 2: $w(e_{ju}) \leq w(e_{jo})$. The left figure is trivial. The right situation cannot occur since the graph is not abstractly planar. }
\label{notnested}
\end{figure}
\vskip 1pt
\end{enumerate}
Combining the seven steps now proves Theorem~\ref{goal}:
Step~(7) shows that a reduced form of the diagram~$D_{R}''$ has no crossings. It follows together with step~(6) that $\G''$ and therefore $\G'$ is trivial. The argument of step~(6) relies on step~(5) and step~(4). By step~(3), which can be performed because of step~(2), $\G$ is also trivial. This proves the theorem by step~(1).
\begin{figure}[h]
	\centering
	\def\svgwidth{500pt}
	 \input{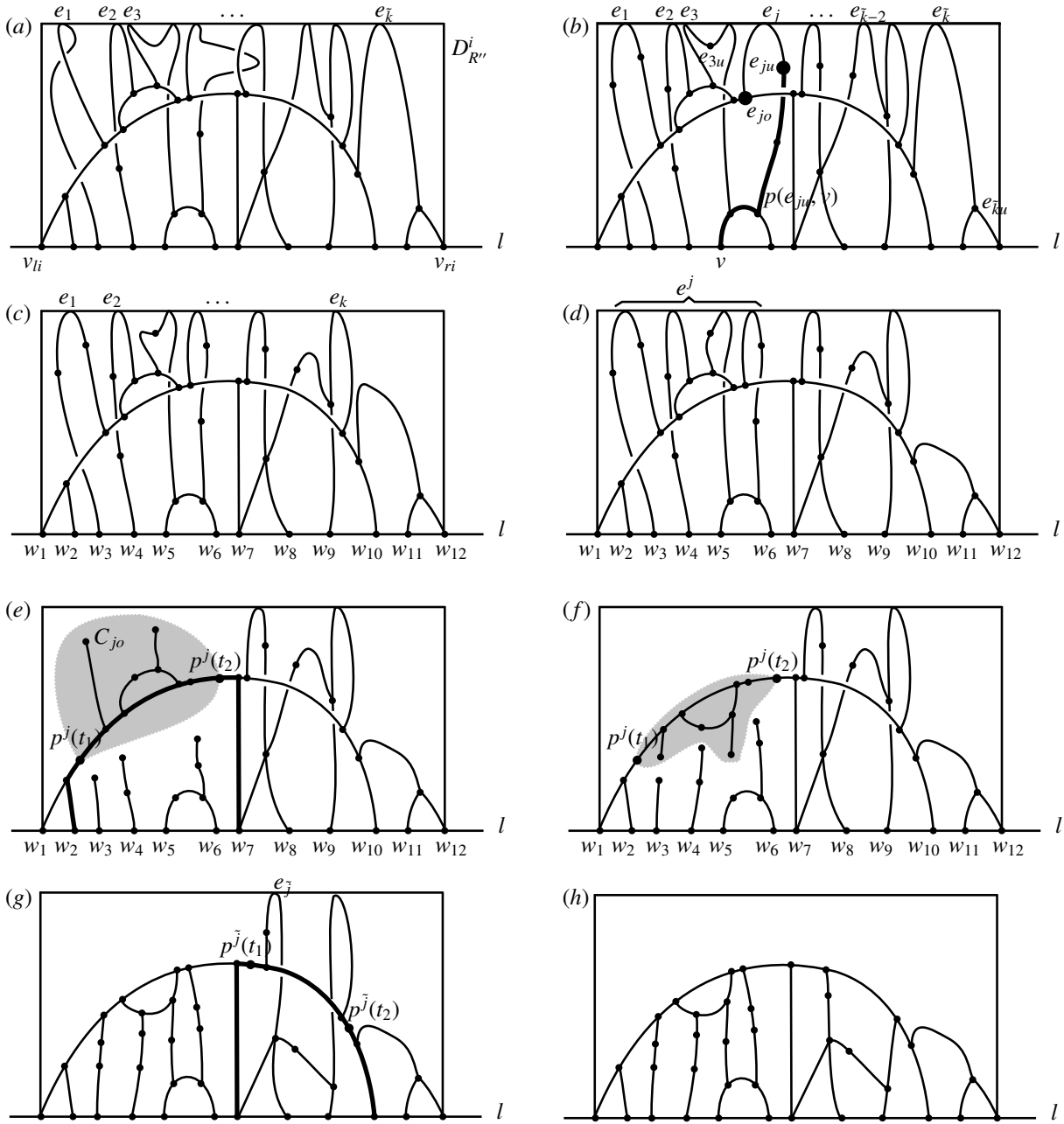}
\caption{Deformation of the diagram $D^{i}_{R''}$ of a spatial fragment~$f_i$ to a crossing free diagram.}
\label{1fragment}
\end{figure}
\end{proof}
\newpage
\begin{cor}[Ravels do not embed on the torus]
Every nontrivial embedding of $\theta_{n}$-graphs on the torus contains a nontrivial knot.
\end{cor}
\begin{proof}
As there exist no pair of disjoint cycles in a $\theta_{n}$-graph, such a graph does not contain a nonsplit link. Since $\theta_{n}$-graphs are planar, the statement of the corollary follows directly from Theorem~\ref{goal}.
\end{proof}
\newpage
\subsection{Alternative proof of Theorem~\ref{goal}:}
The proof of Theorem~\ref{goal} given above can be differently finished using Theorem~\ref{Wu} by Wu~\cite{Wu}. This is a shortcut in the argument but does not give an explicit deformation.
\begin{thm}[Criterion for an embedding of a planar graph to be trivial~\cite{Wu}] \label{Wu} 
The embedding $\G$ of a planar graph $G$ is trivial if and only if every cycle in the spatial graph bounds an embedded disc $D$ whose interior $\mathring{D}$ is disjoint from $\G$.
\end{thm} 
\begin{proof} \textit{(alternative proof of Theorem~\ref{goal})}
Start with step~(1) and step~(2) of the proof that is given in the section above. Recall that the point $P$ is defined as follows: Except elementary cases, there exists a meridian of the torus on which the spatial graph $\mathcal{G}$ is embedded on such that the meridian intersects the spatial graph in only one point. Take this point to be $P$. To apply Theorem~\ref{Wu}, observe that every cycle in $\mathcal{G}$ bounds a disc $D$ which is embedded in $\R^3$ with interior $\mathring{D}$ disjoint from $\mathcal{G}$. This is clearly true for any meridian and for any cycle that does not intersect $P$ since $\mathcal{G} \setminus P$ is embedded on a sphere already. We are left to consider cycles $C$ that run through $P$ for which there exists a natural number $n$ so that the cycle has knot type $T(1,n)$. Let $C_{n}$ be one of these cycles, i.e., $C_{n}$ follows the longitude once and wraps $n$ times around the meridian. We can find an ambient isotopy $i_{n}$ of $\G$ that transforms $C_{n}$ to the longitude $l=i_{n}(C_{n})$ of a new torus $i_{n}(T)$ (not isotopic to $T$), by possibly performing another twist as described in step~(3) of the proof above. Denote the spatial graph that results from this twist by $i_{n}(\G)$. As a longitude bounds a disc internally disjoint from the torus, it follows that the cycle $i_{n}(C_{n})$ bounds a disc internally disjoint from $i_{n}(\G)$. Since ambient isotopies preserve embedded discs and do not pass them through the graph, it follows that the cycle $C_{n}$ in $\G$ bounds a disc internally disjoint from $\G$. For every $n \in \N$,  we can perform such an ambient isotopy of $\G$. This shows that every cycle in $\G$ bounds a disc which is internally disjoint from the spatial graph. As $G$ is planar by assumption, it follows from Theorem~\ref{Wu} that $\G$ is trivial. 
\end{proof}

\begin{rem}
It is not possible to weaken the assumptions of Theorem~\ref{goal} as shown by giving counter examples in~\cite{short}.
\end{rem}




\end{document}